\documentclass[11pt]{amsart}

\usepackage{amsmath, amssymb, amsthm}
\usepackage{graphicx} 
\usepackage{hyperref}
\usepackage[svgnames]{xcolor}

\newtheorem{thm}{Theorem}[section]
\newtheorem{prop}[thm]{Proposition}
\newtheorem{lemma}[thm]{Lemma}

\theoremstyle{definition}
\newtheorem{defn}[thm]{Definition}

\theoremstyle{remark}

\newtheorem{rmk}[thm]{Remark}

\newcommand{\mE}{\mathbb{E}}

\newcommand{\mR}{\mathbb{R}}

\newcommand{\mS}{\mathbb{S}}
\newcommand{\mZ}{\mathbb{Z}}

\newcommand{\e}{\varepsilon}

\newcommand{\len}{\mathrm{len}}
\newcommand{\Lk}{\mathrm{Lk}}
\renewcommand{\H}{\mathrm{H}}
\newcommand{\UT}{\mathrm{T}^1}

\renewcommand{\H}{\mathrm{H}}

\newcommand{\ds}{\displaystyle}

\begin{document}

\title{Left-handed geodesic flow of spheres of revolution}

\author[Dehornoy]{Pierre Dehornoy}
\address{P. Dehornoy, Aix-Marseille Université, CNRS, I2M, Marseille, France,}
\email{pierre.dehornoy@univ-amu.fr}
\urladdr{https://www.i2m.univ-amu.fr/perso/pierre.dehornoy/}

\author[Rechtman]{Ana Rechtman}
\address{A. Rechtman, Institut Fourier and Institut Universitaire de France (IUF),
Universit\'e Grenoble Alpes,
100 rue des math\'ematiques,
38610, Gi\`eres, France\\
}
\email{ana.rechtman@univ-grenoble-alpes.fr}
\urladdr{https://www-fourier.ujf-grenoble.fr/~rechtmaa/}

\begin{abstract}
A flow on a 3-manifold is left-handed if any two ergodic invariant measures have negative linking number. 
We prove that on a 2-sphere of revolution whose curvature is $1/4$-pinched the geodesic flow is left-handed. 
Conversely, for every $\delta\le 1/4$, we construct a 2-sphere whose curvature is $\delta$-pinched and whose geodesic flow is not left-handed. This gives credit to a conjecture of Florio and Hryniewicz that $1/4$ is the optimal pinching constant for left-handedness among arbitrary positively curved 2-spheres. 
\end{abstract}

\maketitle

\tableofcontents

\date{}

\section{Introduction}

The aim of this note is to study the condition of left-handedness for geodesic flows of spheres of revolution, or in other words genus zero surfaces of revolution. 
Left-handedness relies on the notion of linking of pairs of invariant measures, and corresponds to constant sign of these linking numbers.
Since the flows we consider here are integrable, they are left-handed if, roughly speaking, every pair of periodic orbits has negative linking number.

Left-handedness, introduced by É.~Ghys \cite{Ghys}, implies in particular that every finite collection of periodic orbits bounds a Birkhoff section.
We proved that, in negative curvature, left-handedness of the geodesic flows occurs for triangular orbifolds only~\cite{D}. 
Here, we relate left-handedness to the pinching of the curvature: the quotient $\delta(S)=\frac{K_{min}}{K_{max}}\leq 1$ of the extreme values of the Gaussian curvature of the surface $S$. 
We say that the curvature is $\delta$-pinched. 
A.~Florio and U.~Hryniewicz proved that, in general, if the Gaussian curvature of a 2-sphere satisfies that $\delta>0.7225$, then the geodesic flow on its unit tangent bundle is left-handed~\cite{FH}\footnote{Florio and Hryniewicz adopt a convention different than ours for linking numbers of lifts of geodesics to the unitary tangent bundle, which leads to a change of sign. Hence their flows are right-handed when ours are left-handed.
}.
They conjectured that the good bound for $\delta$ should be $1/4$. 
Our main result is that this is the optimal bound among spheres of revolution.

\begin{thm}\label{thm-revolution}
Let $S$ be a sphere of revolution in~$\mR^3$ with the induced metric. 
If $\delta(S)>1/4$, the geodesic flow on~$\UT S$ is left-handed. 

Conversely, for every $\delta\le 1/4$, there exists a $\delta$-pinched sphere of revolution in~$\mR^3$ whose geodesic flow is not left-handed (nor right-handed). 
\end{thm}

The proof relies on a direct computation of the linking numbers between any two closed geodesics and on the fact that the geodesic flow admits an invariant function given by Clairaut's identity, see Section~\ref{ss-clairaut}. 

A ``figure-eight'' geodesic on a surface of revolution is a closed geodesic that turns twice around the axis of revolution and intersects the equator at one point with multiplicity two. 
Lifts of figure-eight closed geodesic and the equator have zero linking number (see Proposition~\ref{prop-linke}). 
We say that the surface has no asymptotic figure-eight geodesic if there exists $a>0$ such that for every closed geodesic $\gamma$ and two points $p_1, p_2$ in the intersection of $\gamma$ with the equator that are consecutive along $\gamma$, the distance between $p_1$ and $p_2$ is at least $a$. 
A figure-eight yields an asymptotic figure-eight, but the converse is not true in general. 
The key point for proving Theorem~\ref{thm-revolution} is the following:

\begin{prop}\label{prop:figure-eight}
Let $S$ be a sphere of revolution in~$\mR^3$, then the geodesic flow on~$\UT S$ is left-handed if and only if there is no asymptotic figure-eight geodesic on~$S$.
\end{prop}

Applying this criterion to ellipsoids of revolution yields in particular

\begin{prop}\label{prop:ellipsoid}
Let $\mathbb{E}_b$ be the ellipsoid of revolution with axes $(1,1,b)$. 
The geodesic flow of $\UT\mathbb{E}_b$ is left-handed if and only if $b<2$. \end{prop}

Since the pinching $\delta(\mE_b)$ is equal to $\frac{1}{b^4}$, this proposition implies the second part of Theorem~\ref{thm-revolution} for $\delta\le\frac1{16}$. 
For $\delta$ in $(\frac1{16}, \frac14]$, we note that the (asymptotic) figure-eight in a pinched ellipsoid lies in a small neighbourhood of the equator and in particular does not use all the freedom brought by the pinching. 
Our proof of the second part of the theorem is obtained by smoothly gluing a neighborhood of the equator on~$\mathbb{E}_{1/\sqrt\delta}$ to a piece of round sphere. 

%This topological condition happens in an arbitrarily small neighborhood of the equator. 
%This observation allows us to construct an example with $\delta=1/4$ whose geodesic flow is not right-handed. 

%We construct in Section~\ref{ss-example} a one-parameter family of metrics indexed by $\epsilon$ on $\mS^2$ such that $\delta=1/4-\epsilon$ and the flow is not right-handed.

%The ellipsoid of revolution $\mE_b$ is defined by the equation in $\mR^3$  $$x^2+y^2+\frac{z^2}{b^2}=1,$$
%If $b\geq 1$, the maximum curvature is attained at the poles, that is when $z=\pm1$ and is equal to $b^2$, while the minimmun curvature is at the equator $\{z=0\}$ and is equal to $1/b^2$. 
%In the case of $\mE_b$, we prove the following result:

\subsection*{Acknowledgements}

We thank A. Florio and U. Hryniewicz, this work started from discussions with them some years ago. 
We also thank Andrés Navas for finding the good change of coordinates in the end of the proof of Proposition~\ref{prop:ellipsoid}. 

This project was (ultimately) supported by the ANR project Anodyn (ANR-24-CE40-5065). 

\section{Linking numbers and left-handed flows}

As mentioned before,  a flow is \emph{left-handed} if all invariant measures have negative linking number. 
In this section, we review some notions and results that allow to make this definition precise. 

\subsection{Linking number of links}
Let $M$ be a 3-dimensional compact manifold that is a rational homology sphere. 
Then the integral homology\linebreak group~$\H_1(M;\mZ)$ is finite, in particular it has an exponent~$e$ which is a positive integer.  
Given an oriented link~$L_1$, its multiple~$eL_1$  bounds a surface~$S_1$. 
Given another link~$L_2$ disjoint from~$L_1$, their \emph{linking number}, denoted by $\Lk(L_1, L_2)$, is the algebraic intersection between $S_1$ and $L_2$, divided by~$e$. 
Thanks to the nullity of~$[L_2]$, it does not depend on the choice of the surface~$S_1$. 
By construction, the linking number is symmetric, bilinear and lies in~$\frac1e\mZ$.

An equivalent construction is obtained as follows. 
The \emph{normal blow-up} of~$L_1$ is the manifold~$M_{L_1}$ obtained by replacing every point~$p$ of~$L_1$ by the unit normal bundle to~$L_1$, denoted by~$N^1_{L_1}(p)$. 
The manifold~$M_{L_1}$ has a toric boundary~$N^1_{L_1}$. 
Because~$M$ is a homology sphere, $\H_1(M_{L_1};\mR)$ is homeomorphic to~$\mR$. 
Denote by~$m_1$ a meridian of~$L_1$, and by $[m_1]$ its class in $\H_1(M_{L_1};\mR)$. 
A multi-curve~$L_2$ has a homology class~$[L_2]$ which is a multiple of~$[m_1]$.  
The \emph{linking number} $\Lk(L_1, L_2)$ is then the unique rational number so that 
$$[L_2]=\Lk(L_1, L_2)[m_1].$$

\subsection{Linking number of invariant measures}\label{sec: linkmeasures}
Assume now that $M$ carries a flow~$(\varphi^t)_{t\in\mR}$ of class~$C^1$. 
Given an ergodic $\varphi^t$-invariant probability measure~$\mu$ and a $\mu$-regular point~$x$, one can pick a sequence~$(t_n)\to+\infty$ so that $\varphi^{t_n}(x)$ tends to~$x$. 
Concatenating the orbit segment $\varphi^{[0,t_n]}(x)$ with the shortest geodesic arc connecting its endpoints, one obtains a closed curve, denoted by~$k_\varphi(x, t_n)$.
Playing the same game with a second ergodic measure~$\nu$ and a $\nu$-regular point~$x'$ not on the orbit of~$x$, one would like to define the linking of $\mu$ and $\nu$ as $\lim_{t_n, t'_n\to\infty}\frac1{t_nt'_n}\Lk(k_\varphi(x, t_n), k_\varphi(x', t'_n))$.
There are difficulties in order to make the above definition rigorous. 
However the flows we consider in this note are integrable, and for this type of flows we shall later prove

\begin{lemma}[linking of ergodic invariant measures]\label{lemma_linkmeasures}
Let $S$ be a sphere of revolution in $\mR^3$ endowed with the induced metric and let $\varphi^t$ be its geodesic flow. 
Let $\mu, \nu$ be two $\varphi^t$-invariant ergodic measures that do not charge the same periodic orbit, and let $x, x'$ be regular points for $\mu$ and $\nu$ respectively, then the limit
\[\lim_{t_n, t'_n\to\infty} \frac1{t_nt'_n}\Lk(k_\varphi(x, t_n), k_\varphi(x', t'_n))\]
exists and does not depend on $x, x', t_n, t'_n$.
\end{lemma}

We then define the \emph{(dynamical) linking} $\Lk^\varphi(\mu, \nu)$ of the invariant measures as the above limit. 

\begin{rmk}
Lemma~\ref{lemma_linkmeasures} is also valid for non integrable flows, we refer to Ghys' article~\cite{Ghys} for a sketch of proof. Completing the sketch is beyond the scope of this note. 
We provide a proof for spheres of revolution in Section~\ref{ss-measures}.
\end{rmk} 

If the invariant measure $\mu$ charges a periodic orbit~$\gamma$, one can perform the following alternative construction. 
We consider the normal blow-up $M_\gamma$ defined in the previous subsection. 
Using its space derivative $D\varphi$, the flow~$\varphi^t$ extends to the boundary torus~$N^1_\gamma$. 
Given an ergodic invariant measure~$\nu$ not charging~$\gamma$, it has an asymptotic cycle $\sigma^\gamma_\nu$ in~$\H_1(M_\gamma; \mR)$. 
Recall that this asymptotic cycle is defined, for $x$ a $\nu$-regular point, as~$\lim_{t\to\infty}\frac1t[k_\varphi(x,t)]$. 
This object is known as Schwartzman's asymptotic cycle~\cite{Schwartzman}. 
Therefore~$\Lk^\varphi(\mu, \nu)$ is equal to~$\frac1{\len(\gamma)}\sigma^\gamma_\nu$ in this case, where $\len(\gamma)$ is the period of the orbit~$\gamma$. 

This second definition of the linking can be naturally extended to a dynamical self-linking number of a periodic orbit: 
assuming that $\mu$ is the ergodic invariant measure charging a periodic orbit~$\gamma$, we can consider a $\varphi^t$-invariant measure~$\mu^\circ$ supported on the boundary torus~$N^1_\gamma$ of~$M_\gamma$ (this measure is not necessarily unique). 
Such an invariant measure has an asymptotic cycle~$\sigma^\gamma_{\mu^\circ}$ in~$\H_1(M_\gamma; \mR)$, and we set~$\Lk^\varphi(\mu, \mu)$ to be~$\frac1{\len(\gamma)}\sigma^\gamma_{\mu^\circ}$. 
Note that~$\sigma^\gamma_{\mu^\circ}$ is the translation number of the extension of the flow to~$N^1_\gamma$ with respect to the Seifert framing, hence it does not depend on the choice of~$\mu^\circ$. 

Summarizing the above discussion, we get

\begin{defn}
Let $\varphi^t$ be an integrable flow on a rational homology sphere~$M$. 
Let $\mu, \nu$ be two $\varphi^t$-invariant ergodic probability measures. 
\begin{enumerate}
    \item[(A)] If $\mu, \nu$ do not charge the same periodic orbit, then we define $\Lk^{\varphi}(\mu, \nu)$ as $\lim_{t_n, t'_n\to\infty}\frac1{t_nt'_n}\Lk(k_\varphi(x, t_n), k_\varphi(x', t'_n))$, where $x, x'$ are regular for $\mu$ and $\nu$ respectively;
    \item[(B)] If $\mu$ charges a periodic orbit~$\gamma$, then we define $\Lk^{\varphi}(\mu, \mu)$ as $\frac1{\len(\gamma)}\sigma^\gamma_{\mu^\circ}$.
\end{enumerate}
\end{defn}

Assume now that $\mu$ is the ergodic invariant measure charging a periodic orbit $\gamma$. Let $\nu_n$ be a sequence of ergodic measures that converge in the \mbox{weak-$*$} sense to $\mu$. 
We claim that $\Lk^{\varphi}(\nu_n,\mu)$ converges to $\Lk^\varphi(\mu,\mu)$. Indeed, there is an invariant measure $\mu^\circ$ in $N_\gamma^1$ that is the weak-$*$ limit of the measure $\nu_n$ in $M_\gamma$. 
Let $x_n$ be a $\nu_n$-regular point, in this setting there is a subsequence of these points that converge to a point $y\in N_\gamma^1$ that is $\mu^\circ$-regular. 
The continuity of the flow  ensures the convergence of the asymptotic cycles.

We can then define left-handedness: 

\begin{defn}\label{defn-lefthanded}
Let $\varphi^t$ be an integrable flow on a rational homology sphere~$M$. 
We say that $\varphi^t$ is \emph{left-handed} if for all $\varphi^t$-invariant ergodic probability measures $\mu, \nu$, one has $\Lk^\varphi(\mu, \nu)<0$. 
\end{defn}

\section{Linking numbers in unit tangent bundles}

As explained in the introduction, for the flows we consider in this note, checking left-handedness will eventually boil down to computing linking of pairs of closed orbits. 
In this purely topological part, we explain how to compute linking numbers for lifts of geodesics to the unit tangent bundles. 

From now on $M$ is the 3-manifold~$\UT\mS^2$. 
An orientation of $\mS^2$ induces an orientation of the fibers by trigonometric order, and the orientation of~$\UT\mS^2$ is obtained by concatenating both.
$\UT\mS^2$ is a circle-bundle over~$\mS^2$ with Euler number $2$. 
It is homeomorphic to~$\mR P^3$. 
It is a rational homology sphere, with~$\H_1(\UT \mS^2, \mZ)\simeq \mZ/2\mZ$. 
Therefore the linking number of any pair of disjoint links is a number in~$\frac12\mZ$.

\subsection{Linking of simple disjoint curves}
We start with the linking of any two fibers:

\begin{lemma}\label{lemma-linkfibers}
Given two distinct points~$p_1, p_2$ on~$\mS^2$, the linking number of their oriented fibers~$\Lk(\UT p_1$, $\UT p_2)$ is~$-\frac12$. 
\end{lemma}

\begin{rmk}
As mentioned before, we have a different orientation convention than from Florio and Hryniewicz~\cite{FH}. The orientation chosen above disagrees with the one chosen by Florio and Hryniewicz. 
\end{rmk}

\begin{proof}
To compute the linking number with the fiber of $p_1$, we claim that there is a natural surface bounded by $-2\UT p_1$. 
Indeed there exists a $C^1$ vector field~$X_1$ on~$\mS^2$ that vanishes at~$p_1$ only. 
By the Poincar\'e-Hopf formula, the index of~$X_1$ at~$p_1$ is $\chi(\mS^2)=2$.
The vector field $X_1$ is not of norm 1, but defines a direction at each point of $\mS^2\setminus \{p_1\}$. 
Hence, lifting the direction defined by~$X_1$ to~$\UT \mS^2$, we obtain a topological surface~$S_{X_1}\subset \UT \mS^2$. 
Orienting~$S_1$ by lifting the orientation of~$\mS^2$, its oriented boundary is then~$-2\UT p_1$, see Figure~\ref{fig-linkfibers}. 
The surface~$S_{X_1}$ intersects once the fiber of every point distinct from~$p_1$, since at each point $p\in \mS^2$ different from $p_1$ the vector field~$X_1$ is non zero. 
This intersection is positive, hence equal to $1$. Therefore one has $\Lk(-2\UT  p_1, \UT  p_2)=1$, so that $\Lk(\UT  p_1, \UT  p_2)=-\frac12$.
\end{proof}

\begin{figure}[htb]
\includegraphics[width=.7\textwidth]{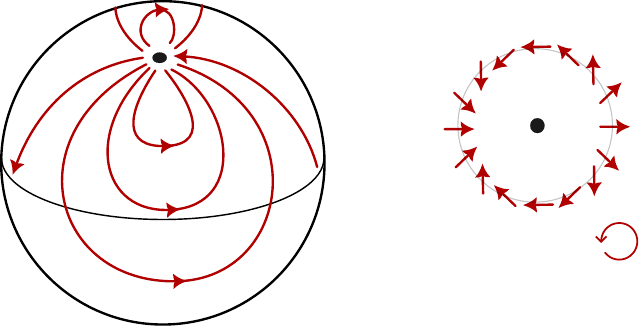}
\caption{On the left, a vector field~$X_1$ on~$\mS^2$ that only vanishes at a point~$p_1$. On the right, when following a circle around~$p_1$, the vector field~$X_1$ makes $+2$ turns, so that its index is~$+2$. Considering the orientation of the surface~$S_{X_1}$ coming from the orientation of~$\mS^2$ (bottom right), one sees that its oriented boundary consists in following the lift of~$X_1$ around~$p_1$ in the negative direction, hence the oriented boundary of~$S_{X_1}$ is~$-2\UT p_1$.}
\label{fig-linkfibers}
\end{figure}

%In the proof above, we made an orientation choice for the surface that is the orientation induced by $\mR^3$ of a sphere embedded in this space. This choice differs from the on in \cite{FH}, thus out linking numbers will be negative. 

From this one can also compute the linking of two disjoint simple curves. 
Given $\gamma_1, \gamma_2$ two disjoint $C^1$ oriented, simple, closed curves in~$\mS^2$, there is an oriented annulus $A\subset\mS^2$ whose oriented boundary is $\pm \gamma_1\pm \gamma_2$. 
If $A$ is bounded by $\gamma_1-\gamma_2$ or by $-\gamma_1+\gamma_2$ we say that the orientations of~$\gamma_1$ and $\gamma_2$ \emph{coincide}. 

\begin{prop}\label{prop-lksimple}
Given $\gamma_1, \gamma_2$ two disjoint $C^1$ oriented, simple, closed curves in~$\mS^2$, 
\begin{itemize}
    \item if the orientations of~$\gamma_1$ and~$\gamma_2$ coincide, then~$\Lk(\vec\gamma_1, \vec\gamma_2)=\frac12$,
    \item otherwise one has~$\Lk(\vec\gamma_1, \vec\gamma_2)=-\frac12$.
\end{itemize}
\end{prop}

\begin{proof}
Denote by $A$ the annulus bounded by~$\pm\gamma_1\pm\gamma_2$ in~$\mS^2$. 
Also pick a point~$p_1$ in the component of~$\mS^2\setminus A$ bounded by~$\pm\gamma_1$, and a point~$p_2$ in the component of~$\mS^2\setminus A$ bounded by~$\pm\gamma_2$. 

If $\partial A=\gamma_1-\gamma_2$, then $\vec\gamma_1$ is isotopic to~$\UT p_1$ in the complement of~$\vec\gamma_2$. Similarly $\vec\gamma_2$ is isotopic to~$-\UT p_2$ in the complement of~$\UT p_1$. 
Therefore one has $\Lk(\vec\gamma_1, \vec\gamma_2)=\Lk(\UT p_1, -\UT p_2)=\frac12$.

The case~$\partial A=\gamma_2-\gamma_1$ is treated similarly, and one obtains $\Lk(\vec\gamma_1, \vec\gamma_2)=\Lk(-\UT p_1, \UT p_2)=\frac12$.

If $\partial A=\gamma_1+\gamma_2$, then $\vec\gamma_1$ is isotopic to~$\UT p_1$ in the complement of~$\vec\gamma_2$, and $\vec\gamma_2$ is isotopic to~$\UT p_2$ in the complement of~$\UT p_1$. 
Therefore one has $\Lk(\vec\gamma_1, \vec\gamma_2)=\Lk(\UT p_1, \UT p_2)=-\frac12$.

And finally if $\partial A=-\gamma_1-\gamma_2$, then $\vec\gamma_1$ is isotopic to~$-\UT p_1$ in the complement of~$\vec\gamma_2$, and $\vec\gamma_2$ is isotopic to~$-\UT p_2$ in the complement of~$\UT p_1$. 
Therefore one has $\Lk(\vec\gamma_1, \vec\gamma_2)=\Lk(-\UT p_1, -\UT p_2)=-\frac12$.
\end{proof}

%By Jordan's theorem, each of them divides the sphere into two discs. 
%Denote by $D_1$ the disc in~$\mS^$2 positively bounded by~$\gamma_1$, and define $D_2$ positively bounded by $\gamma_2$ similarly.  
%We say that the orientations of $\gamma_1$ and $\gamma_2$ \emph{coincide} if the symmetric difference $D_1\Delta D_2$ is non empty and none of its connected components is bounded by either $\gamma_1+\gamma_2$ or an oriented loop in $\gamma_1\cup \gamma_2$.

\subsection{Seifert's desingularization and Reidemeister II moves}\label{ss-linkinglift}

%Two oriented $C^1$ curves~$\gamma_1, \gamma_2$ on~$\mS^2$ are in general position if: (1) the self-intersection points are double transverse self-intersections points, (2) $\gamma_1$, $\gamma_2$ are transverse to one another with only double intersection points and that the intersection points do not coincide with self-intersection points. 
An oriented multi-curve on a surface is the union of a finite number of (non necessarily disjoint) oriented closed curves.
An oriented $C^1$ multi-curve~$\gamma$ on~$\mS^2$ is in general position if all multiple points are double transverse self-intersection points. Its lift to~$\UT\mS^2$ is denoted by~$\vec\gamma$. 

Two oriented $C^1$ multi-curves~$\gamma_1, \gamma_2$ on~$\mS^2$ are in general position if both curves are in general position and all intersection points are double transverse intersection points. 
Their lifts $\vec\gamma_1, \vec\gamma_2$ in~$\UT\mS^2$ form a link. 
In this section we explain how to compute the linking number~$\Lk(\vec\gamma_1, \vec\gamma_2)$. 

\begin{defn}
Given a multi-curve~$\gamma$ on~$\mS^2$ in general position, a multi-curve obtained from~$\gamma$ by desingularizing every double point of~$\gamma$ in a small enough disc, see Figure~\ref{fig-seifert}, is called a \emph{Seifert desingularization} of~$\gamma$. 
It is denoted by $\gamma^S$. 
\end{defn}

\begin{figure}[htb]
xz\includegraphics[width=.4\textwidth]{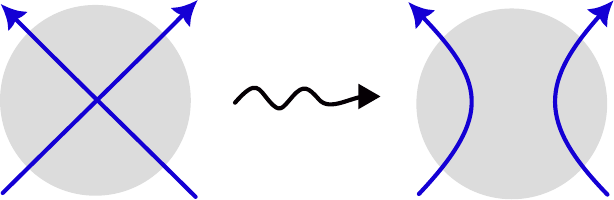}
\caption{The Seifert desingularization of a double point. One removes the double point and connects the remaining strands is the unique way that respects the orientation. The operation is made in a small disc (in gray) that is disjoint from all other parts of the curve and all other discs where the desingularization is performed.}
\label{fig-seifert}
\end{figure}

The Seifert desingularization of a multi-curve is a simple multi-curve ({\it i.e.}, it has no self-intersection points). 
The next lemma reduces the computation of the linking number of lifts of closed curves to the linking number of lifts of simple multi-curves. 

\begin{lemma}
Given two oriented $C^1$ multi-curves $\gamma_1, \gamma_2$ on~$\mS^2$, denote by $\gamma^S_1, \gamma^S_2$ some Seifert desingularizations. 
Then one has 
$$\Lk(\vec\gamma_1, \vec\gamma_2)=\Lk(\vec\gamma^S_1, \vec\gamma^S_2).$$
\end{lemma}

\begin{proof}
We claim that, for $(i,j)=(1,2)$ or $(2,1)$, there is a cobordism between $\vec\gamma_i$ and~$\vec\gamma_i^S$ that is disjoint of $\vec\gamma_j$, namely a surface whose boundary is $\vec\gamma_i^S -\vec\gamma_i$. 
Indeed, as sets, $\gamma_i$ and $\gamma_i^S$ differ only in a neighborhood of the self-intersection points of $\gamma_i$, these self-intersection points correspond to fibers in $\UT \mS^2$ that intersect $\vec\gamma_i$ at two points, hence the boundary of a tubular neighborhood of the fiber intersects in four points the curve $\vec\gamma_i$. 
The lift $\vec\gamma_i^S$ connects these four points in the other possible way that respects the orientations, the arcs in these connections are contained  in the tubular neighborhood of the fiber.
We have then a local cobordism. 
Locally this cobordism  is disjoint from $\vec\gamma_j$, since there are no intersection point between $\gamma_1$ and $\gamma_2$ in a neighborhood of the self-intersection points of both curves.
\end{proof}

\medskip

We now define an analog of the Reidemeister II move for oriented simple closed curves on $\mS^2$ and we explain how this operation affects the linking number of their lifts to $\UT \mS^2$. 

Suppose that $\gamma_1$ and $\gamma_2$ are two simple non-oriented curves in $\mS^2$, in general position, such that $\gamma_1\cap \gamma_2\neq \emptyset$. 
%The number of intersection is even, denote it by~$2k$. 
Consider two intersection points $p_1$ and $p_2$ that are consecutive along both curves, then there is a topological disc $D\subset \mS^2$ bounded by the closed curve given by the segments of $\gamma_1$ and $\gamma_2$ between $p_1$ and $p_2$ that does not contain other intersection points. 
%Assume that $D$ is a bigon, meaning that the segments we consider are embedded. 

The \emph{Reidemeister II move}, moves one of the curves, say $\gamma_2$, to another curve $\gamma'_2$ that coincides with $\gamma_2$ outside an open neighborhood $U$ of $D$, see Figure~\ref{fig-RII}. 
It reduces the number of intersection points by $2$. For $t\in [0,1]$ denote by $\gamma_2^t$ a continuous 1-parameter family of closed curves achieving a deformation from $\gamma_2$ to~$\gamma'_2$: $\gamma_2^t$ coincides with $\gamma_2$ outside~$U$ for every $t$; $\gamma_2^0=\gamma_2$ and $\gamma_2^1=\gamma'_2$, see Figure~\ref{fig-RII}.

Now assume that $\gamma_1$ and $\gamma_2$ are oriented, the orientation of $\gamma_2$ induces an orientation on $\gamma_2^t$ for every $t$. 
There are two cases for the disc $D$: either the orientations of $\gamma_1$ and $\gamma_2$ make the bigon bounding $D$ oriented, or not. 

\begin{lemma}\label{lemma-RII}
Consider two closed simple multi-curves~$\gamma_1, \gamma_2$ on~$\mS^2$ in general position. 
Assume that $\gamma'_2$ is obtained from $\gamma_2$ by applying a Reidemeister II move with respect to~$\gamma_1$ in a bigon~$D$. 
\begin{itemize}
    \item If the orientations of $\gamma_1$ and $\gamma_2$ make the bigon bounding $D$ oriented, then one has $\Lk(\vec \gamma_1, \vec \gamma'_2)=\Lk(\vec \gamma_1, \vec \gamma_2)$.
    \item Otherwise one has $\Lk(\vec \gamma_1, \vec \gamma'_2)=\Lk(\vec \gamma_1, \vec \gamma_2)+1$.
\end{itemize}
\end{lemma}

\begin{figure}[htb]
\includegraphics[width=.7\textwidth]{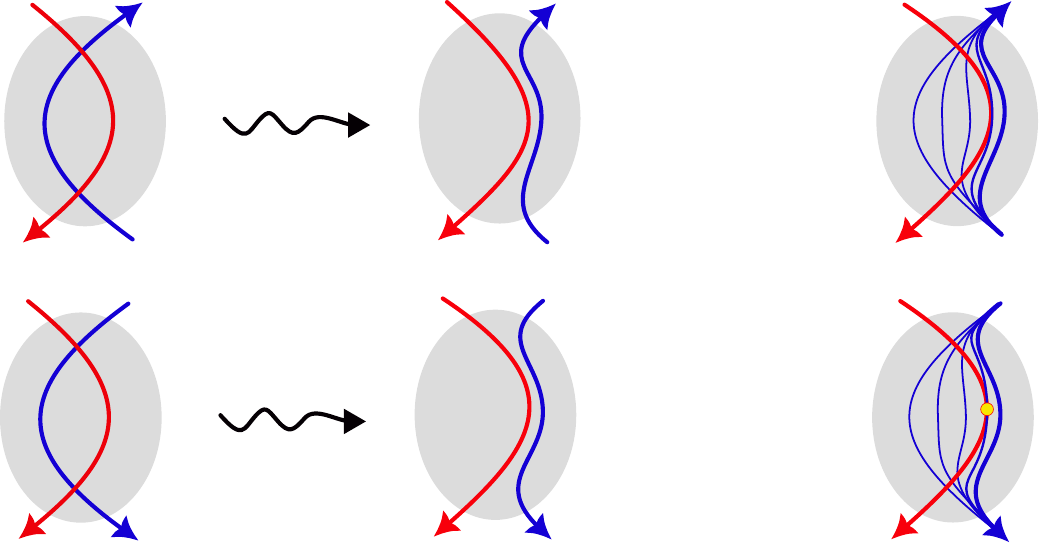}
\caption{The Reidemeister II move on a pair of oriented curves. The move itself does not depend on the orientations of the curves, but the family~$\gamma_2^t$ (on the right) lifts to a surface that intersects or not ~$\gamma_1$, depending on whether the orientations match or not.}
\label{fig-RII}
\end{figure}

\begin{proof}
Consider the curves $\vec \gamma_1$ and $\vec \gamma_2^t$ for $t\in [0,1]$. 

In the first case, all vectors positively tangent to $\gamma_2^t$ are disjoint from the vectors positively tangent to~$\gamma_1$. 
Then they induce a cobordism from $\vec \gamma_2$ to~$\vec \gamma'_2$ that does not intersect~$\vec\gamma_1$. 
Therefore one has $\Lk(\vec \gamma_1, \vec \gamma_2)=\Lk(\vec \gamma_1, \vec \gamma'_2)$. 

In the second case, there is a $t\in (0,1)$ such that $\gamma_1$ and $\gamma^t_2$ have a common tangent vector, meaning that $\vec \gamma_1$ and $\vec \gamma_2^t$ intersect. 
This intersection point is negative, hence $\Lk(\vec \gamma_1, \vec \gamma_2)=\Lk(\vec \gamma_1, \vec \gamma'_2)-1$.
\end{proof}

Lemma~\ref{lemma-RII} yields a procedure to compute the linking number of the lifts of two simple closed curves: starting from the two curves, one iteratively applies Reidemeister II moves to make them disjoint. 
At every step one removes 0 or 1 to their linking number, depending on whether the orientations match. 
In the end there remains two disjoint simple curves, the linking of whose lifts is given by Proposition~\ref{prop-lksimple}

The situation we will encounter is even simpler since no complicated meanders can appear for geodesics on spheres of revolution.

%Given two oriented closed curves $\gamma_1$ and $\gamma_2$ in $\mS^2$, we describe an algorithm to compute the linking number~$\Lk(\vec\gamma_1, \vec\gamma_2)$, resumed in Figures~\ref{F:Simple} and \ref{F:Seifert}. General position is a standing hypothesis in what follows, even if we forget to mention it.

\begin{prop}\label{prop-lksimple2}
Let $\gamma_1$ and $\gamma_2$ be two oriented simple closed curves on~$\mS^2$. 
Assume that $\gamma_1\cap\gamma_2$ consists of $2k$ points with $k>0$. 
Label them in a increasing order $1, 2, \dots, 2k$ when travelling along~$\gamma_1$. 
\begin{itemize}
%\item If $\gamma_1\cap\gamma_2=\emptyset$, let $A\subset \mS^2$ be the annulus bounded by $\pm\gamma_1\pm \gamma_2$. If the orientations of $\gamma_1$ and $\gamma_2$ coincide, then $\Lk(\vec \gamma_1, \vec \gamma_2)=\frac 12$. If not, $\Lk(\vec \gamma_1, \vec \gamma_2)=-\frac 12$.
\item  %If $\gamma_1\cap\gamma_2\neq \emptyset$ the intersection contains $2k$ points for some $k\in \mN$. 
If when travelling along $\gamma_2$ one encounters the intersection points in the ordering $1, 2, \dots, 2k$, then one has $\Lk(\vec \gamma_1, \vec \gamma_2)=\frac 12-k$. 
\item If when travelling along $\gamma_2$ one encounters the intersection points in the ordering $2k, 2k-1, \dots, 1$, then one has $\Lk(\vec \gamma_1, \vec \gamma_2)=-\frac 12$.
\end{itemize}
\end{prop}

\begin{proof}
%Finally, if the two curves~$\gamma_1, \gamma_2$ intersect, 
One can use $k$  Reidemeister~II moves to separate the two curves~$\gamma_1$, $\gamma_2$, and obtain two disjoint oriented curves~$\gamma_1, \gamma'_2$. 
When using this move, either the orientations disagree and $\Lk(\gamma_1,\gamma_2)=\Lk(\gamma_1,\gamma_2')=-\frac 12$ or the orientations agree, and this moves adds $k$ to the linking number of the lifts in~$\UT \mS^2$, thus
$$\Lk(\gamma_1,\gamma_2)=\Lk(\gamma_1,\gamma_2')-k=\frac 12-k.$$
\end{proof}

We summarise the linking number of the lifts of two simple closed curves (that do not form complicated meanders) in Figure~\ref{fig-simple}.

%We can now address the case of any two oriented closed (multi-)curves $\gamma_1$ and $\gamma_2$ in general position. 
%We first desingularize all double points in order to have two collections of disjoint simple closed curves. 
%Then for every pair made of one curve in the first collection and one curve in the second collection, one applies Proposition~\ref{prop-lksimple} or~\ref{prop-lksimple2}. 

%We first apply Seifert's algorithm to each curve, introduced in \cite{}\footnote{Je ne sais pas si on a besoin d'une reference ou pas...

%On peut peut-être dire qu'on lisse les points doubles, avec une toute petite figure}, to reduce to the case of simple curves in $\mS^2$.  This algorithm consists in changing the curves at self intersections points: we first erase the curve in a neighborhood of the the intersection point and then connect the four endpoints created respecting the orientations and obtaining a (possibly disjoint) oriented curve with one less intersection point. 
%This turns each curve~$\gamma_i$ into a collection~$\gamma_i^S$ of simple closed curves on~$\mS^2$, for $i=1,2$.  By construction, $\gamma_i^S$ has a well defined orientation given by the orientation of $\gamma_i$.

\begin{figure}
    \begin{picture}(340,130)(0,0)
    \put(-10,5){\includegraphics[width=1\textwidth]{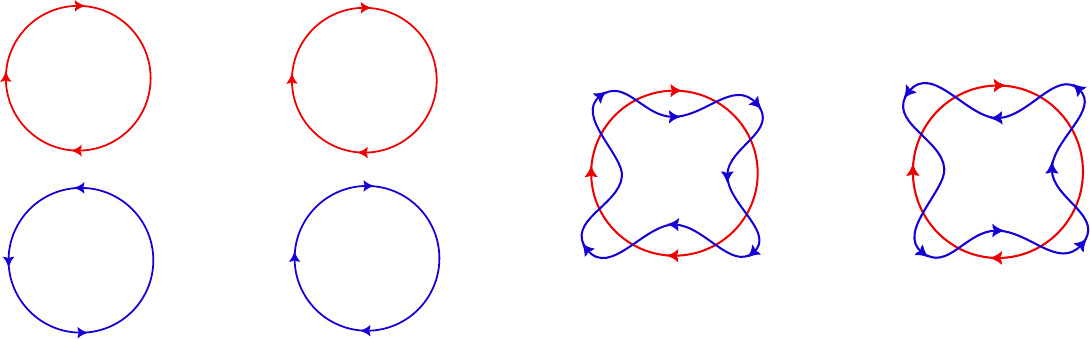}}
    \put(-3,58){$\frac12$}
    \put(83,58){$-\frac12$}
    \put(200,13){$\frac12-k$}
    \put(310,13){$-\frac12$}
    \end{picture}
    \caption{The linking number~$\Lk(\vec \gamma_1, \vec \gamma_2)$ of the lifts of two simple closed curves~$\gamma_1, \gamma_2$ on~$\mS^2$. 
On the left, when the curve do not intersect, they can be parallel (linking ~$+\frac12$) or anti-parallel (linking~$-\frac12$). 
On the right, when the curves intersect, they can be made disjoint using $k$ Reidemeister II moves. These moves alter the linking number when the resulting curves are parallel, and do not alter it when the curves are anti-parallel.
}
    \label{fig-simple}
\end{figure}
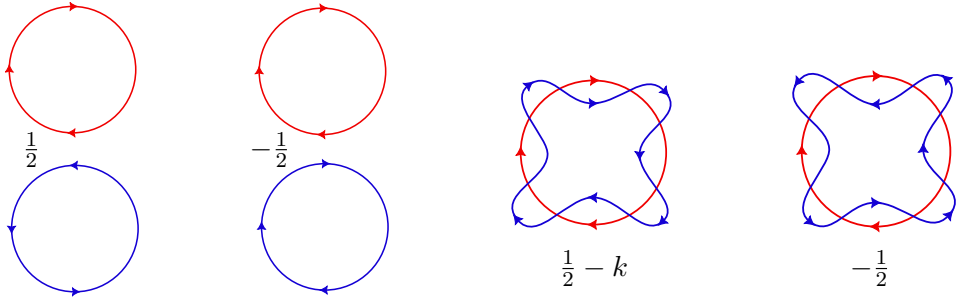

\section{Positively curved surfaces of revolution and their geodesic flows}\label{sec-revolution}

We now recall some properties of surfaces of revolution and of their geodesic flows. 

\subsection{Surfaces of revolution, their geodesics and Clairaut's relation}\label{ss-clairaut}
Let $a<b$ be two real numbers. 
Consider two $C^2$ functions $f,g:[a,b]\to\mR$ such that $f(v)>0$ for all $v\in (a,b)$ and $f(a)=f(b)=0$. 
Consider the surface of revolution $S_{f,g}$ in $\mR^3$ defined as 
$$S_{f,g}:=\{(f(v)\cos u, f(v)\sin u, g(v)); \quad v\in [a,b] \quad \mbox{and} \quad 0\leq u<2\pi\}.$$
Note that $f$ gives the distance of a point to the axis of revolution. 
The Gaussian curvature of $S_{f,g}$ at a point is given by 
$$K(u,v)=-\frac{g'(g'f''-g''f')}{f((f')^2+(g')^2)^2},$$
and, when $f'^2+g'^2=1$, we get $K(u,v)=-\frac{f''}{f}$. In this setting, the curvature is strictly positive if $f''(v)<0$ for all $v\in(a,b)$.

A curve on $S_{f,g}$ corresponds to a curve in the $(u,v)$-plane that can be parametrised as $(u(t),v(t))$. Denoting by $u'$ and $v'$ the derivatives of these functions with respect to $t$, the curve is geodesic if and only if it satisfies:
\begin{eqnarray}\label{eq-geodesic1}
u''+\frac{2ff'}{f^2}u'v' &=0\label{eq-geodesic2}\\
v''-\frac{ff'}{(f')^2+(g')^2}(u')^2+\frac{f'f''+g'g''}{(f')^2+(g')^2}(v')^2 &=0.
\end{eqnarray}

If $u(t)$ is a constant function, that is, we consider a meridian on $S_{f,g}$, the first equation is trivially satisfied and the second equation can be deduced from the first fundamental form of the surface (see page 161 of \cite{DoCarmo}). 
Hence, meridians are geodesics.

Now if $v(t)$ is a constant function, the first equation gives $u''=0$, hence $u'$ is a non-zero constant function, while the second equation gives
$$-\frac{ff'}{(f')^2+(g')^2}(u')^2=0.$$
This implies $ff'=0$ and hence if $f\neq0$ then $f'=0$. 
Therefore a parallel is a geodesic if and only if it corresponds to a critical point of~$f$.

Now observe that Equation~ \eqref{eq-geodesic1} can be written as
$$(f^2u')'=f^2u''+2ff'u'v'=
0,$$
hence $f^2u'$ is constant. 
On the other hand, if at a certain point a geodesic forms an angle $\theta$, $-\pi<\theta\leq \pi$ with respect to the parallel (oriented by $u$) at the given point,  then
$$\cos \theta=\frac{\langle x_u,x_uu'+x_vv'\rangle}{|x_u|}=fu'$$
where $x_u=(-f(v)\sin u, f(v)\cos u, 0)$ and $x_v=(f'(v)\cos u, f'(v)\sin u, g'(v))$ form a basis for the tangent space of $S$. 
Since $f$ is the radius of the parallel, we obtain Clairaut's relation:
$f\cos\theta$ is constant along a geodesic.

\subsection{Invariant tori and invariant measures for spheres of revolution}\label{ss-measures}

We start with a sphere of revolution~$S_{f,g}$ as in the previous section. 
It consists of those points of the form~$p(u,v)=(f(v)\cos u, f(v) \sin u, g(v))$ with $v\in[a,b]$ and $u\in[0,2\pi)$. 
We call $u$ the argument of $p$ and $f$ the radius of~$p$, that is, the distance from $p$ to the axis of revolution. 
For $(p,w)\in\UT S$ and $p$ not a pole, define %$r(p)$ as the distance from $p$ to the axis of revolution and 
$\theta(w)\in (-\pi,\pi]$ as the angle between~$w$ and the tangent vector to the (positive) parallel. 
For $\alpha\in\mR$, define 
$$T_\alpha:=\{(p,w)\in\UT S\,|\, f(p)\cos(\theta(w))=\alpha\}.$$ 
Clairaut's relation implies that, for every $\alpha$, $T_\alpha$ is invariant by the geodesic flow. 

We now assume that the curvature of $S_{f,g}$ is strictly positive, so that $f'$ can be zero only once. 
Thus there is only one geodesic that lies on a plane orthogonal to the axis of revolution. 
We call it the \emph{equator}, denoted by~$e$. 
We also denote by $e_+$ the oriented geodesic obtained by following the equator in the positive direction ({\it i.e.,} with $u$ increasing), and by $e_-$ the oriented geodesic obtained by following the equator in the negative direction. 
Observe that $S\setminus \{e\}$ has two components, the north and south hemisphere. 
We denote by $r_e$ the radius of~$e$. 
For $\alpha=\pm r_{e}$, the set $T_\alpha$ is a circle that coincides with one of the two lifts of~$e$, namely the orbits~$\vec e_+$ and~$\vec e_-$ of the geodesic flow.
For $\alpha\in(-r_e, r_{e})$, $T_\alpha$ is a torus, foliated by the orbits of the geodesic flow. 

Denote by~$\ell_\alpha$ the parallel in~$T_\alpha$ of the form
$$\{(p,w)\in\UT S; f(p)=\alpha, \cos(\theta(w))=1\}$$ and by $m_\alpha$ the oriented meridian in~$T_\alpha$ of the form
$$\{(p,w)\in\UT S; u(p)=\mathrm{cst}, f(p)\cos(\theta(w))=\alpha\}.$$ 
The pair $(\ell_\alpha, m_\alpha)$ forms a basis of the first homology group of~$T_\alpha$. 
%One can parametrize~$T_\alpha$ with two parameters $u,w\in\mS^1$, where $u$ is the longitudinal coordinate and $w$ the meridian coordinate.

The restriction of the geodesic flow to~$T_\alpha$ is invariant by rotation in the longitudinal direction. 
When integrating along an invariant measure, the geodesic flow has an asymptotic cycle of the form $(x_\alpha, y_\alpha)$ in the basis~$(\ell_\alpha, m_\alpha)$. 
The invariance by rotation implies in particular that $(x_\alpha, y_\alpha)$ does not depend on the chosen invariant measure on~$T_\alpha$. 
Geometrically, $x_\alpha$ is the inverse of the period of the flow in the longitudinal direction and it has the sign of~$\alpha$. 
Also, $y_\alpha/x_\alpha$ is the average number of times the corresponding geodesic crosses the equator positively over one longitudinal period. 

Note that $T_\alpha$ is foliated by periodic orbits if and only the ratio $y_\alpha/x_\alpha$ is rational.
If $y_\alpha/x_\alpha$ is irrational, the restriction of the flow to~$T_\alpha$ is conjugated to a linear flow. 
In this case, for $p\in T_\alpha$ and~$t$ large enough, %a long piece of orbit of type $\varphi^{[0,T]}(p)$ closed by a short curve 
the curve $k_\varphi(p,t)$ is a torus knot on~$T_\alpha$ of type $(p_\alpha^t,q_\alpha^t)$, where $q_\alpha^t/p_\alpha^t$ is a rational approximation of~$y_\alpha/x_\alpha$. 

We can now explain why, in this context, the linking of two ergodic invariant measures exists.  

\begin{proof}[Proof of Lemma~\ref{lemma_linkmeasures}]
Assume that $\mu, \nu$ are two ergodic $\varphi^t$-invariant measures, supported by~$T_\alpha, T_\beta$ respectively. 
For $x$ a $\mu$-regular point and $x'$ a $\nu$-regular point, and for sequences $t_n, t'_n$ such that $\varphi^{t_n}(x)\to x$ and $\varphi^{t'_n}(x')\to x'$, the knots $k_\varphi(x, t_n)$ and $k_\varphi(x', t'_n)$ lie in (a neighborhood of) $T_\alpha$ and $T_\beta$ respectively. 
Since $\mu$ and $\nu$ do not charge the same periodic orbit, the points $x$ and $x'$ can be chosen so that they belong to different orbits. This implies that the knots $k_\varphi(x, t_n)$ and $k_\varphi(x', t'_n)$ are disjoint and form a link.

The classes of $k_\varphi(x, t_n)$ and $k_\varphi(x', t'_n)$ in the first homology groups of the corresponding tori are roughly $t_n(x_\alpha, y_\alpha)$ and $t'_n(x_\beta, y_\beta)$. 
Therefore $\frac1{t_nt'_n}\Lk(k_\varphi(x, t_n), k_\varphi(x', t'_n))$ is roughly the linking of the curves $(x_\alpha, y_\alpha)$ on $T_\alpha$ and $(x_\beta, y_\beta)$ on~$T_\beta$. 
This concludes the proof. 
\end{proof}

\section{Left-handedness for positively curved spheres}

We put together the ingredients of the two previous sections to estimate the linking number of any two ergodic invariant measures, and then prove the main results. 

\subsection{Reduction to linking with the equator}

We show that for a positively curved sphere of revolution, left-handedness can be checked by computing the linking with the equator only. 
We use the notations of Section~\ref{sec-revolution}. 

For $\beta>0$ and $\alpha<\beta$, the two invariant tori $T_\alpha, T_\beta$ are disjoint. 
Let $\mu$ be an ergodic measure whose support is contained in $T_\alpha$. 
A curve on~$T_\beta$ may be homotoped onto a curve on~$T_{r_e}$ in the complement of~$T_\alpha$. 
More precisely, if the curve has slope type $(x_\beta,y_\beta)$ in the basis~$(\ell_\beta, m_\beta)$, it may be homotoped onto $x_\beta$ times the curve $T_{r_e}=\vec e_+$ in the complement of~$T_\alpha$. 
Therefore the linking of two invariant measures $\mu, \nu$ supported by~$T_\alpha, T_\beta$ has the same sign as $\Lk^\varphi(\mu, \vec e_+)$, where we abuse of the notation and denote the ergodic invariant measure supported by $\vec{e}_+$ with the same sign. 

By a similar argument, if $\beta<0$ and $\alpha>\beta$, the linking of two invariant measures $\mu, \nu$ supported by~$T_\alpha, T_\beta$ has the same sign as $\Lk^\varphi(\mu, \vec e_-)$. 
Thus we get

\begin{lemma}\label{lemma-LkEquator}
The geodesic flow on a sphere of revolution~$S$ is left-handed if and only if there exists $a>0$ such that for every closed geodesic~$\gamma$ on~$S$, one has $\frac1{\len(\gamma)}\Lk(\vec e_+, \vec \gamma)<-a$ and $\frac1{\len(\gamma)}\Lk(\vec e_-, \vec \gamma)<-a$. 
\end{lemma}

\begin{proof}
First assume that the flow is left-handed. 
By Definition~\ref{defn-lefthanded} this means that the linking of any two invariant measures is negative. 
Assume for a contradiction that there is no positive~$a$ satisfying the statement. 
Then, up to reversing their orientations, there is a sequence of oriented periodic geodesics~$\gamma_n$ such that $\frac1{\len(\gamma_n)}\Lk(\vec e_+, \vec \gamma_n)$ tends to~$0$. 
Recall from Section~\ref{sec: linkmeasures} that $\Lk(\vec e_+, \cdot)$ is continuous in the weak-* topology. 
By extracting a subsequence we find a invariant measure whose linking with $\vec e_+$ is zero, a contradiction.     

We prove the converse implication. 
Let $\mu$ and $\nu$ be two ergodic measures of the geodesic flow of $S$. 
Assume first that $\mu$ and $\nu$ do not charge the same periodic orbit. 
Then the above discussion implies that $\Lk^\varphi(\mu, \nu)$ equals the linking of either $\mu$ or $\nu$ with either $\vec e_+$ or $\vec e_-$, hence it is smaller than $-a$.

Now assume that $\mu=\nu$ is a measure supported by a periodic orbit $\vec \gamma$ of the geodesic flow. Let $\alpha\in (-r_e, r_e)$ be such that $\vec \gamma\subset T_\alpha$, then $y_\alpha/x_\alpha$ is a rational number and we can choose a sequence $\beta_n$ of numbers in $(-r_e,r_e)$ such that $y_{\beta_n}/x_{\beta_n}\neq y_\alpha/x_\alpha$ for all $n$ and forms a monotone sequence that converges to $y_\alpha/x_\alpha$. Consider the manifold $\UT S_{\vec \gamma}$ obtained by blowing up $\vec \gamma$ and an ergodic  measure $\mu^\circ$ of the induced flow, whose support is contained in $\partial \UT S_{\vec \gamma}$ and is the limit of measures $\nu_n$ supported in the tori $T_{\beta_n}$ (observe that the flow in $\UT S_{\vec \gamma}$ leaves these tori invariant, with the exception of $T_\alpha$). Then
$$\Lk^{\varphi}(\mu, \mu)= \lim_{n\to \infty} \Lk^\varphi(\nu_n,\mu^\circ)\leq -a<0.$$
We conclude that the flow is left-handed. 
\end{proof}

\subsection{Linking with the equator and aysmptotic figure-eights}

In what follows, we explain why, in the case of spheres of revolution, the geodesic flow is left-handed if and only if there are no asymptotic figure eight geodesics (Proposition~\ref{prop:figure-eight}). 
%To prove that if there are no asymptotic figure eight geodesics, the flow is left-handed we have to check that the linking between any pair of distinct invariant measures is negative and that the self-linking number of each closed geodesic is also negative. 

%We first compute  linking numbers between closed geodesics and the equator. 

%For this section we consider a surface of revolution $S\subset \mR^3$ with strictly positive curvature. Let $e$ be the equator, as introduced in Section~\ref{ss-clairaut}. Recall that the equator is the unique geodesic, up to orientation, that lies in a plane orthogonal to the axis of revolution. We denote by $e_+$ and $e_-$ the two oriented geodesics obtained from orienting $e$: in the parametrisation of Section~\ref{ss-clairaut}, $e_+$ is oriented in the trigonometric sense in the $xy$-plane of $\mR^3$. 
As before, $S$ denotes a sphere of revolution with positive curvature. 
Its equator $e$ lifts to two closed orbits of the geodesic flow~$\vec e_+$ and $\vec e_-$. 
Consider an oriented geodesic $\gamma$ that is different from the equator. 
Then $\gamma$ intersects $e_+$ and $e_-$. 
Denote by $\theta_\gamma\in (0,\pi)$ the angle between $\gamma$ and $e_+$ at a point in which $\gamma$ crosses $e_+$ from the south to the north hemisphere. 
Clairaut's identity implies that at each crossing with the equator the angle between the geodesic and $e_+$ is equal to $\pm\theta_\gamma$. 

Firstly, the intersection between $e_+$ and $\gamma$ has an even number of points that is denote by $2q_\gamma$ (here the orientation of $\gamma$ is not important). 
Secondly, $\gamma$ makes a certain number of turns around the rotation axis that we call $p_\gamma$. 
We call the pair $(p_\gamma, q_\gamma)$ the \emph{type} of~$\gamma$. 
Observe that $p_\gamma\in \mZ$ and the sign is determined by the trigonometric direction in the $xy$-plane or, equivalently, by the sign of~$\cos(\theta_\gamma)$. 
Moreover, $p_\gamma=0$ implies that $\gamma$ is a meridian, then $q_\gamma=1$ and $\theta_\gamma=\pi/2$.

With the notations of Section~\ref{sec-revolution}, the periodic orbit~$\vec\gamma$ lies in the invariant torus $T_\alpha$ with $\alpha=r_e\cos(\theta_\gamma)$. 
The latter is foliated by periodic orbits all parallel to~$\vec\gamma$. 
In the $(\ell_\alpha, m_\alpha)$-basis, they have a rational direction $(x_\alpha, y_\alpha)$ which is colinear with $(p_\gamma, q_\gamma)$.

\begin{prop}\label{prop-linke}
The linking number between $\vec{e}_+$ and $\vec{e}_-$ equals $-1/2$. 

Let $\gamma$ be a closed oriented geodesic of $S$ of type $(p_\gamma,q_\gamma)\in \mZ^2$ with $q_\gamma\geq 1$.
\begin{enumerate}
\item If $\theta_\gamma\in (0,\pi/2)$ or equivalently $p_\gamma>0$, then
$$\Lk(\vec e_+,\vec \gamma)=\frac {p_\gamma}2-q_\gamma, \qquad \mbox{and} \qquad \Lk(\vec e_-,\vec \gamma)=-\frac {p_\gamma}2.$$
\item If $\theta_\gamma\in (\pi/2,\pi)$ or equivalently $p_\gamma<0$, then
$$\Lk(\vec e_+,\vec \gamma)=-\frac{|p_\gamma|}{2}, \qquad \mbox{and} \qquad \Lk(\vec e_-,\vec \gamma)=\frac{|p_\gamma|}{2}-q_\gamma.$$
\item If $\theta_\gamma=\pi/2$ or $p=_\gamma0$ we have
$$\Lk(\vec e_+,\vec \gamma)=\Lk(\vec e_-,\vec \gamma)=-\frac 12.$$
\end{enumerate}
\end{prop}

\begin{proof}
In $\UT S$, we can push $\vec e_-$, in the complement of $\vec e_+$, to a curve $\sigma_-$ so that the projection to $S$ of $\sigma_-$ lies in the north hemisphere of $S$. Then we obtain two disjoint oriented simple closed curves with opposite orientations as in the first case of Proposition~\ref{prop-lksimple}, hence $\Lk(\vec e_+,\vec e_-)=-1/2$.

Now assume $\theta_\gamma\in (0,\pi/2)$ or, equivalently, $p_\gamma>0$. 
Since $\gamma$ is a geodesic, its self-intersection points are transverse as well as the intersection points between $\gamma$ and $e_+$. 
Following the ideas at the end of the Section~\ref{ss-linkinglift}, we first move $e_+$ up to the north hemisphere, if needed, so that $e_+$ and $\gamma$ are in general position: that is to avoid having self-intersections points of $\gamma$ coincide with intersection points between $\gamma$ and $e_+$. 
In $\UT S$ we can move $\vec e_+$ in the complement of $\vec \gamma$.

Then we apply Seifert's smoothing to $\gamma$ to obtain a collection of oriented simple closed curves $\gamma^S$, as in Section~\ref{ss-linkinglift}. 
The multi-curve $\gamma^S$ has $p_\gamma$ connected components and only one intersects $e_+$, along $2q_\gamma$ points as in Figure~\ref{F:Seifert}. 
Applying Proposition~\ref{prop-lksimple2} we obtain
$$\Lk(\vec e_+, \vec \gamma)=\Lk(\vec e_+,\vec \gamma^S)
=\frac{(p_\gamma-1)}{2}+\left(\frac 12-q_\gamma\right)=\frac{p_\gamma}{2}-q_\gamma;$$
and $\Lk(\vec e_-, \vec \gamma)=-\frac{p_\gamma}{2}$.

The case $\theta_\gamma\in (\pi/2,\pi)$ or $p_\gamma<0$ is analogous, except that $\gamma$ turns in the direction of $e_-$.
Changing the roles of $e_+$ and $e_-$ we deduce the formulas. 

If $\theta_\gamma=\pi/2$, $\gamma$ is an oriented simple close curve in $S$ and  we can directly apply Proposition~\ref{prop-lksimple2}.
\end{proof}

\begin{figure}
\includegraphics[width=.7\textwidth]{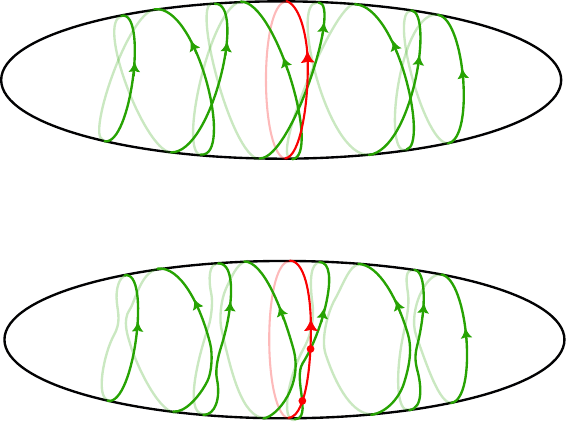}
\caption{An orbit~$\gamma$ of the geodesic flow of type $(p_\gamma, q_\gamma)=(8,1)$ (top, green). 
Its Seifert desingularization (bottom, green) consists of $p_\gamma=8$ parallels circles, $p_\gamma-1=7$ of which have linking $+\frac12$ with the oriented equator $\vec e_+$ (red) and one of which has linking $+\frac12-q_\gamma=-\frac12$ with it. Thus the linking number between the green and red curves is equal to 3. }
\label{F:Seifert}
\end{figure}

\begin{rmk}
Using the same ideas, we can compute the linking number between two arbitrary closed oriented geodesics of $S$. 
For this we deform one geodesic into a cover of the equator as in the discussion in Section~\ref{ss-measures}. 
We do not need these formulas in the sequel, but we include them without proof for the interested reader:
%\begin{prop}\label{prop-link2}

Let  $\gamma_1$ and $\gamma_2$ be two closed oriented geodesics of types $(p_1, q_1)$ and $(p_2,q_2)$ respectively. 
\begin{enumerate}
\item If $p_1p_2\neq0$, assume that $|\pi/2-\theta_1|\geq |\pi/2-\theta_2|$, then
$$
\Lk(\vec \gamma_1, \vec \gamma_2)=\left\{\begin{matrix}
|p_1|\left(\frac{|p_2|}{2}-q_2\right) & p_1p_2>0\\
-\frac{|p_1p_2|}{2} & p_1p_2<0
\end{matrix}\right.$$
\item If $p_1=p_2=0$, then $\Lk(\vec \gamma_1, \vec \gamma_2)=-\frac 12$.
\item If $p_2=0$ and $p_1\neq 0$, then $\Lk(\vec \gamma_1, \vec \gamma_2)=-\frac{|p_1|}{2}$.
\end{enumerate}
%\end{prop}
\end{rmk}

%\begin{proof}
%Assume $p_1p_2\neq 0$ then neither $\gamma_1$ nor $\gamma_2$ are meridians. The hypothesis $|\pi/2-\theta_1|> |\pi/2-\theta_2|$ implies that $\gamma_2$ is closer to the vertical direction when crossing the equator. Equivalently, the farthest north point of $\gamma_2$ is closer to the north pole than any point in $\gamma_1$. Apply Seifert's smoothing to $\gamma_1$. We get a collection of $|p_1|$-curves that turn once around the axis of revolution.  Again this transformation moves $\vec \gamma_1^S$ in the complement of $\vec \gamma_2$ and hence does not changes the linking number. Now we can isotope  $\vec \gamma_1^S$ in the complement of $\vec \gamma_2$ to $p_1$ copies of $\vec e_+$ if $p_1>0$ or to $|p_1|$ copies of $e_-$ if $p_1<0$. Applying Proposition~\ref{prop-linke} we obtain the formula in item (1).

%If $p_i=0$ then $\gamma_i$ is a closed oriented simple curve and hence we can apply Proposition~\ref{prop-lksimple2} to conclude in cases (2) and (3). 
%\end{proof}

\bigskip 

Observe that all linking numbers in Proposition~\ref{prop-linke} are strictly negative if $q_\gamma-\frac{|p_\gamma|}{2}>0$ for every oriented closed geodesic~$\gamma$. 
The condition is equivalent to $\frac{|p_\gamma|}{q_\gamma}<2$. 

\begin{defn}\label{defn:fig8}
A closed geodesic of type $(\pm 2,1)$ is a \emph{figure eight geodesic}.

We say that $S$ has \emph{no asymptotic figure-eight geodesic} if \[\sup_{\gamma~\mathrm{closed}}\left(\frac{|p_\gamma|}{q_\gamma}\right)<2.\]
\end{defn}

Let $\gamma$ be  a closed geodesic. 
Observe that $q_\gamma$ is equal to the number of segments of $\gamma$ that are contained in the norther hemisphere (or southern hemisphere). 
The hypothesis $\sup\left(\frac{|p_\gamma|}{q_\gamma}\right)<2$ is then equivalent to saying that there is a strictly positive lower bound to the distance between the endpoints of any segment of $\gamma$ contained in a hemisphere and that this bound is independent of $\gamma$.

Proposition~\ref{prop-linke} implies that the linking number between $\vec e_\pm$ and one of the possible orientations of a figure eight geodesic is zero. 
Hence a figure eight geodesic is an obstruction to left-handedness. 

%In order to study the linking, we need to understand what happens to every orbit of the geodesic flow and not just to periodic ones. 
%The geodesic flow is integrable, that is, $\UT S\setminus \{\vec e_\pm\}$ is foliated by invariant tori and on each torus the vector field is linear. The slopes of the vector field on the tori cover an interval that is symmetric with respect to zero. 

%Consider as before the geodesic $e_+\in S$. The lifts of the oriented geodesics that make an angle $\theta$ with $e_+$ when crossing from the south to the north hemisphere, form an invariant torus in $\UT S$. We can endow each torus with meridian, longitude coordinates in such a way that on the torus of closed orbits that are lifts of type $(p,q)$ geodesics the vector field has slope $m=\frac{p}{q}$. Then $m=0$ is the case of the meridians of $S$ and coincides with the curves $(0,1)$ in longitude-meridian coordinates in the torus. Observe that the closed connected orbits of the geodesic flow that are a lift of a closed geodesic of type $(p,q)$ is a $(p,q)$-torus knot, hence $p$ and $q$ are relatively prime. 

%We now place ourselves in the case of a Riemannian metric such that the surface has no figure eight geodesic, that is equivalent to saying that for every closed oriented geodesics  $q-\frac{|p|}{2}>0$ or $-2<\frac pq<2$. 

We can now prove the following result stated in the introduction that summarizes the above discussion.

\newtheorem*{prop:figure-eight}{\bf Proposition \ref{prop:figure-eight}}
\begin{prop:figure-eight}
Let $S$ be a sphere of revolution in~$\mR^3$, then the geodesic flow on~$\UT S$ is left-handed if and only if there is no asymptotic figure-eight closed geodesic on~$S$.
\end{prop:figure-eight}

\begin{proof}
%By definition we have that $\sup\left(\frac{|p|}{q}\right)<2$ among the closed geodesics of $S$. 
By Lemma~\ref{lemma-LkEquator} the flow is left-handed if and only if there exists $a>0$ such that for every closed geodesic $\gamma$,
$$\frac1{\len(\gamma)}\Lk(\vec e_+, \vec \gamma)<-a \qquad \mbox{and} \qquad \frac1{\len(\gamma)}\Lk(\vec e_-, \vec \gamma)<-a.$$
Denote by $(p_\gamma,q_\gamma)$ the type of~$\gamma$, as defined in the beginning of this section. 
Assume $p_\gamma>0$, the  case $p_\gamma<0$ being analogous. Proposition~\ref{prop-linke} implies 
$$\frac1{\len(\gamma)}\Lk(\vec e_+, \vec \gamma)
=\frac1{\len(\gamma)}\left(\frac{p_\gamma}{2}-q_\gamma\right)<-a$$
and 
$$\frac1{\len(\gamma)}\Lk(\vec e_1, \vec \gamma)
=\frac1{\len(\gamma)}\left(-\frac{p_\gamma}{2}\right)<-a.$$
Since $q_\gamma\geq 1$, these inequalities are equivalent to
$$\frac{2a}{q_\gamma}\len(\gamma)<\frac {p_\gamma} {q_\gamma}<2-\frac{2a}{q_\gamma}\len(\gamma).$$
Taking into account the cases $p_\gamma=0$ and $p_\gamma<0$, we get that the flow is left-handed if and only if there exists $a>0$ such that
$$0\leq \frac{|p_\gamma|}{q_\gamma}<2-\frac{2a}{q_\gamma}\len(\gamma).$$
Since $a>0$, this is equivalent to $\sup\left(\frac{|p_\gamma|}{q_\gamma}\right)<2$ among closed geodesics.
\end{proof}

\subsection{Proof of the first part of Theorem~\ref{thm-revolution}.}
In this section, we prove that if $\delta>1/4$, the geodesic flow is left-handed. 
It follows from a result of Abbondandolo, Bramham, Hryniewicz and  Salom\~{a}o \cite{ABHS}, that is a consequence of a result by Toponogov.

Let $S$ be a sphere of revolution satisfying $\delta(S)>1/4$. 
Consider an oriented geodesic $\gamma$ different from $e$. 
We parametrize $\gamma$ in such a way that $\gamma(0)\in e$. 
Let $t_1>0$ be the first point at which $\gamma(t_1)\in e $. 
We now copy

\begin{lemma}[Lemma 3.9 in \cite{ABHS}]\label{lemma2.9}
Assume that $S$ is a sphere endowed with a Riemannian metric such that $\delta>1/4$. Then the geodesic arc $\gamma(t)$ for $t\in[0,t_1]$ is injective.
\end{lemma}

This lemma is valid for spheres that are not surfaces of revolution. The positiveness of the curvature defines an equator $e$. Observe that  the arc of geodesic $\gamma(t)$ for $t\in[0,t_1]$ is by definition contained in the closure of one of the hemispheres of $S$ and since it is injective, $\gamma(0)\neq \gamma(t_1)$. In particular, $\gamma$ cannot be a figure eight closed geodesic.
Furthermore, the distance along $e$ between $\gamma(0)$ and $\gamma(t_1)$ admits a strictly positive lower bound that is independent of $\gamma$. Thus $S$ has no asymptotic figure eight geodesic. Proposition~\ref{prop:figure-eight} implies that the geodesic flow of $S$ is left-handed.

%Back to surfaces of revolution, consider a connected closed geodesic of type $(p,q)$ and the open geodesic arcs in $\gamma \setminus e$. Without loss of generality we assume that $p>0$. Since $\delta>1/4$,  Lemma~\ref{lemma2.9} implies that each arc of $\gamma\setminus e$ makes strictly less than one turn around the axis of revolution and its closure contains exactly two points in $e$. Hence there are exactly $2q$ segments in $\gamma\setminus e$ and $|p|<2q$. Hence for every closed geodesic we have that $\frac{|p|}{q}<2$. Proposition~\ref{prop-asymplink} implies that the geodesic flow is left-handed.

\subsection{Left-handedness for ellipsoids of revolution}\label{sec: ellipsoid}

In this section we consider the ellipsoid $\mE_b$ defined by the equations in $\mR^3$
$$x^2+y^2+\frac{z^2}{b^2}=1,$$
with $b\geq 1$ and the $z$-axis is the rotation axis. We prove in this section that its geodesic flow is left-handed if $b<2$:

\newtheorem*{prop:ellipsoid}{\bf Proposition \ref{prop:ellipsoid}}
\begin{prop:ellipsoid}
Let $\mathbb{E}_b$ be the ellipsoid of revolution with axes $(1,1,b)$. 
The geodesic flow of $\UT\mathbb{E}_b$ is left-handed if and only if $b<2$.
\end{prop:ellipsoid}

\begin{proof}

A parametrisation of $\mE_b$ as a surface of revolution, see Section~\ref{ss-clairaut}, is given by
\[
\begin{array}{rcl}
x(u,v) & = & \cos v \cos u\\
y(u,v) & = & \cos v \sin u\\
z(u,v) & = & b \sin v, 
\end{array}
\]
with $u\in[0,2\pi)$ and $v\in [-\pi/2,\pi/2]$.
Thus its curvature is 
$$K(u,v)=\frac{b^2}{(\sin^2v+b^2\cos^2v)^2}.$$
If $v=0$, that is along the equator, we have that the curvature is equal to $1/b^2$ and is the minimum curvature if $b> 1$, the maximum curvature if $b<1$. The other extremal value of the curvature is attained at the poles, that is when $v=\pm \pi/2$ and is equal to $b^2$. Hence $\delta(\mE_b)=1/b^4$.

\medskip

Consider an arc~$\gamma$ of geodesic that meets the equator at $(x,y,z)=(1,0,0)$ of longitude $u=0$. We now compute the coordinates of the first point of intersection of $\gamma$  with the equator. Let~$\theta_0$ be the angle between the equator $e_+$ and $\gamma$. By symmetry, we assume $0<\theta_0\leq \pi/2$.  Let $\alpha_0=\pi/2-\theta_0$, the angle between the meridian and $\gamma$.
For simplicity we set $\sigma_0=\sin\alpha_0$. 
For a point on~$\gamma$, we recall that $u$ is its longitude and $f$ is its radius (distance to the $z$-axis). 
We also call $c$ the distance along the ellipsoid to the equator and~$\alpha$ the angle between~$\gamma$ and the meridian towards the north pole, see Figure~\ref{F:Coor}. 
We will also use the $z$ coordinate of a point in $\gamma$.
%Note that we have $\ds r^2+\frac{h^2}{b^2}=1$. 

\begin{figure}
\begin{picture}(130,250)(0,0)
\put(0,0){\includegraphics[width=.3\textwidth]{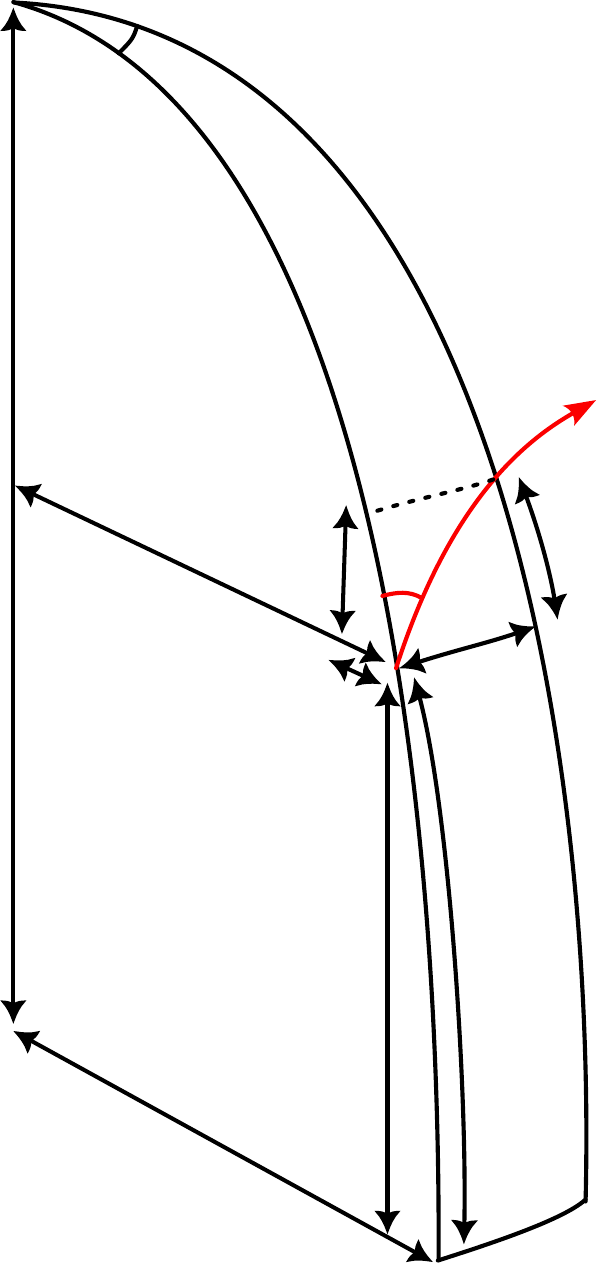}}
\put(-5,150){$b$}
\put(35,114){$f$}
\put(60,50){$h$}
\put(85,50){$c$}
\put(70,123){$\alpha$}
\put(59,96){$df$}
\put(50,125){$dh$}
\put(101,125){$dc$}
\put(8,214){$du$}
\put(78,101){$f\,du$}
\end{picture}
\caption{An arc of geodesic (red).}
\label{F:Coor}
\end{figure}
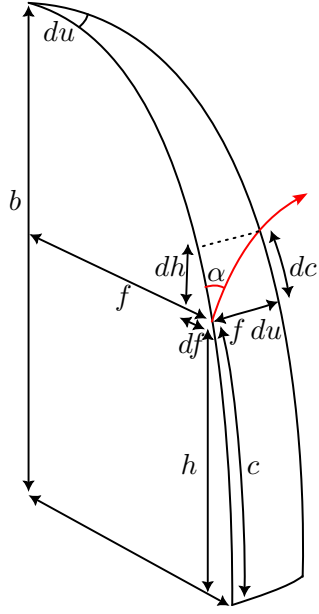

%\begin{figure}
%\includegraphics[width=.5\textwidth]{CoorEllipsoid.jpg}
%\caption{An arc of geodesic (red).}
%\label{F:Coor}
%\end{figure}

By Clairaut's relation the quantity~$f\sin(\alpha)$ is constant along~$\gamma$. 
In particular the radius $f_{max}$ of the highest point on~$\gamma$ corresponds to~$\alpha=\frac\pi2$, hence $f_{max}=\sigma_0$. 
Note that $\gamma$ starts from the equator, goes north for some time during which $f$ decreases from 1 to $f_{max}$. 
Let $u_{max}$ be the longitude of the highest point. 
By symmetry, the first-return longitude coordinate on the equator is then twice~$u_{max}$.
%The number of turns one needs for coming back on the equator with the same angle is then~$\frac2\pi u_{max}$.

To compute $u_{max}$ we will compute the value of~$\ds\frac{du}{dr}$ and integrate it from $1$ to~$f_{max}$.
Looking at Figure~\ref{F:Coor}, we start with the relation
$\frac{f\,du}{dc}=\tan(\alpha)$.
Clairaut's relation yields $\sin(\alpha) = \frac{\sigma_0}f$, so we get
\begin{equation}\label{eq1}
 \frac{f^2\,du^2}{dc^2}=\tan^2(\alpha)=\frac{\sin^2(\alpha)}{1-\sin^2(\alpha)}=\frac{\frac{\sigma_0^2}{f^2}}{1-\frac{\sigma_0^2}{f^2}}=\frac{\sigma_0^2}{f^2-\sigma_0^2}.
 \end{equation}
Since we are on $\mE_b$ we have $\ds f^2+\frac{z^2}{b^2}=1$, so that $\frac{z^2}{b^2}=1-f^2$. 
Differentiating we get $\frac{z}{b^2}dz=-f\,df$, and so 
\[dz^2=\frac{b^4f^2}{z^2}df^2=\frac{b^2f^2}{1-f^2}df^2.\]
By Pythagoras' theorem we have 
\[dc^2=dz^2+df^2=\left(1+\frac{b^2f^2}{1-f^2}\right)df^2=\frac{1+(b^2-1)f^2}{1-f^2}df^2\]
Putting it back into~\eqref{eq1} we get
\[\frac{du^2}{df^2}=\frac{(1+(b^2-1)f^2)\sigma_0^2}{f^2(1-f^2)(f^2-\sigma_0^2)}.\]
Integrating this we obtain
\begin{equation}\label{eq2}
u_{max}=\int_{\sigma_0}^1\sqrt{\frac{(1+(b^2-1)f^2)\sigma_0^2}{f^2(1-f^2)(f^2-\sigma_0^2)}}df=\sigma_0\int_{\sigma_0}^1\frac{1}f\sqrt{\frac{(1+(b^2-1)f^2)}{(1-f^2)(f^2-\sigma_0^2)}}df.
\end{equation}
The integrand diverges at~$\sigma_0$ and~$1$, but it is integrable at both points. 

%When~$\sigma_0\to 1$, we have~$\ds u_{max}\sim \int_{\sigma_0}^1\frac{b}{\sqrt{(1-r^2)(r^2-\sigma_0^2)}}dr\sim \frac\pi2b$, which means that when we shoot slightly above the equator, we make $b$ half-turns before coming back to the equator.
%On the other hand, when~$\sigma_0\to 0$, we shoot almost toward the north pole, so we should come back to the equator with longitude coordinate near to $\pi$. Numerical computations confirm in that case that $\ds u_{max}\to\pi/2$ as it should be. 

Set $a=b^2-1$. 
We perform the change of variable $v^2=\frac{(f^2-\sigma_0^2)a}{1-\sigma_0^2}$, then
$$
u_{max}  =  \int_0^{\sqrt{a}} a\sigma_0 \sqrt{\frac{1+(1-\sigma_0^2)v^2+a\sigma_0^2}{(a-v^2)((1-\sigma_0^2)v^2+a\sigma_0^2)^2}}du.
$$
We make a second change of variable $v=\sqrt{a}\cos t$ to obtain
$$
u_{max}  =  \int_0^{\pi/2} a\sigma_0\frac{\sqrt{1+a\sigma_0^2+(1-\sigma_0^2)a\cos^2 t}}{a\sigma_0^2+(1-\sigma_0^2)a\cos^2 t}dt.
$$
We can now compute the limit when $\sigma_0$ tends to one, that, by the dominated convergence theorem, is equal to
\begin{eqnarray*}
\int_0^{\pi/2} \left(\lim_{\sigma_0\to 1}a\sigma_0\frac{\sqrt{1+a\sigma_0^2+(1-\sigma_0^2)a\cos^2 t}}{a\sigma_0^2+(1-\sigma_0^2)a\cos^2 t}\right)dt & = & \int_0^{\pi/2}\sqrt{1+a}dt\\
&= & \frac{\pi}{2}b.
\end{eqnarray*}

Therefore, $u_{max}$ is bounded away from $\pi$ if and only if $b<2$, so that $\mE_b$ has no asymptotic figure-eight geodesic if and only if $b<2$. 
\end{proof}

\subsection{Proof of the second part of Theorem~\ref{thm-revolution}}\label{ss-example}

Fix $\delta$ in $(0,1]$. 
We now explain how to construct a sphere of revolution $S_\delta$ whose curvature is $\delta$-pinched.
Roughly speaking it is obtained by flattening a round sphere near the equator into a piece of an ellipsoid of revolution with the desired curvature. 
The construction works for every $\delta\le 1$, but it is interesting only for $\delta<1/4$ in which case one finds two geodesics who link positively and for $\delta=1/4$ in which case the self-linking of the equator vanishes.

As in Section~\ref{ss-clairaut} we describe $S_\delta$ with a radius function $f$, and an auxiliary height function $g$. 
The construction is depicted on Figure~\ref{F:Huit}. 
For convenience we will use an extra function $b$ that will parametrize the curvature. 
Fix $\e>0$ small enough. 
For $s\in[-\pi,\pi]$ we define 
\[f(s)=b(s)\cos(s), \quad \mathrm{and}\quad g(s)=\sin(s), \]
where $b$ is a $C^\infty$ bump function that satisfies
\begin{equation*}
b(s)=
    \begin{cases}
        \frac1{\sqrt\delta}  & \text{if } s \in [-\e,\e]\\
        1 & \text{if } s \in [-1,-2\e]\cup[2\e,1],\\
    \end{cases}
\end{equation*}
and $b$ interpolates smoothly between $\frac1{\sqrt\delta}$, and $1$ on $[-2\e,-\e]\cup[\e,2\e]$. 
Denote by~$S_{\delta}$ the sphere of revolution induced by the pair~$(f,g)$. 
It is made by connecting an annular neighborhood of the equator of~$\mE_{\frac1{\sqrt\delta}}$ with polar caps cut from the round sphere~$\mE_1$. 
Notice that, at a point of $S_\delta$ corresponding to a parameter~$s$, the sphere is osculated by an ellipsoid of the form~$\mE_{b(s)}$, so that its curvature is everywhere between~{$\delta$} and~$1$, and so it is $\delta$-pinched. 

We now explain why the geodesic on~$S_\delta$ is not left-handed for $\delta<1/4$. 
Indeed, in this case one has $b(s)>2$ for~$s\in[-\e,\e]$. 
By {Equation~\eqref{eq2}}, there exists a figure-eight closed geodesic~$\gamma$ in the $\e$-neighborhood of the equator on~$S_\delta$. 
By {Proposition~\ref{prop-linke}}, the oriented lift of~$\gamma$ in~$\UT S_\delta$ has zero linking with the lift of the equator. 
Therefore the geodesic flow on~$S_\delta$ is not left-handed. 

\begin{figure}
\includegraphics[width=.5\textwidth]{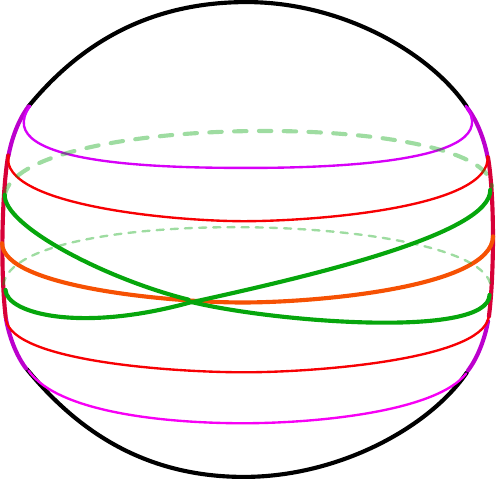}
\caption{The surface~$S_\delta$ and the geodesic figure-eight (green) whose lift has zero linking with the equator (orange).}
\label{F:Huit}
\end{figure}

Finally, we consider the case $\delta=1/4$. 
The sphere $S_{1/4}$ has a neighborhood of its equator which is a neighborhood of the equator of~$\mE_{1/2}$. 
The computation at the end of the proof of Proposition~\ref{prop:ellipsoid} ensures that there is an asymptotic figure-eight geodesic in this neighborhood, hence the geodesic flow is not left-handed. 

This concludes the proof of Theorem~\ref{thm-revolution}.

\end{document}